 \documentclass[11pt,3p,compress]{elsarticle}
\usepackage{graphicx}

\graphicspath{{../pdf/}{../jpeg/}}
\DeclareGraphicsExtensions{.pdf,.jpeg,.png}
\usepackage{booktabs}	
\usepackage{epstopdf}
\usepackage{color}
\usepackage{pifont}
\usepackage{tikz}
\usepackage{hyperref}
\usepackage{eurosym}

\usepackage{pgfplots}
\usepackage{setspace}
\usepackage{rotating}
\usepackage{circuitikz}
\usepackage{bbm}

\usepackage{amssymb,amsmath,url}

\usepackage{algpseudocode}

\usepackage{lipsum}
\usepackage{algorithm} 
\usepackage{enumitem}

\algrenewcommand\algorithmicrequire{\textbf{Input:}}
\algrenewcommand\algorithmicensure{\textbf{Output:}}
\algrenewcommand\algorithmicforall{\textbf{For}}

\newtheorem{theorem}{Theorem}
\newtheorem{lemma}{Lemma}
\newtheorem{proposition}{Proposition}
\newtheorem{corollary}{Corollary}

\allowdisplaybreaks
\newtheorem{definition}{Definition} 
\newtheorem{proof}{Proof}

\newcommand{\R}{\mathbb{R}}
\newcommand{\Z}{\mathbb{Z}}

\newcommand*{\QEDA}{\hfill\ensuremath{\blacksquare}}

\makeatother
\makeatletter
\def\ps@pprintTitle{ 
 \let\@oddhead\@empty
 \let\@evenhead\@empty
 \def\@oddfoot{\textit{Working Paper}\hfill}%
 \let\@evenfoot\@oddfoot}
\makeatother

\begin{document}
\begin{frontmatter}

\title{{Performance guarantees of forward and reverse greedy algorithms for minimizing nonsupermodular nonsubmodular functions on a matroid}}

\author[a1]{Orcun Karaca\corref{cor1}}
\ead{okaraca@ethz.ch}
\author[a1]{Daniel Tihanyi}
\ead{tihanyid@ethz.ch}
\author[a2]{Maryam Kamgarpour}
\ead{maryamk@ece.ubc.ca}
\address[a1]{Automatic Control Laboratory, D-ITET, ETH Z{\"u}rich, Switzerland}
\address[a2]{Electrical and Computer Engineering Department, University of British Columbia, Vancouver, Canada}
\cortext[cor1]{Corresponding author}
\begin{abstract}
This letter studies the problem of minimizing increasing set functions, or equivalently, maximizing decreasing set functions, over the base of a matroid. This setting has received great interest, since it generalizes several applied problems including actuator and sensor placement problems in control theory, multi-robot task allocation problems, video summarization, and many others.  We study two greedy heuristics, namely, the forward and the reverse greedy algorithms. We provide two novel performance guarantees for the approximate solutions obtained by these heuristics depending on both the submodularity ratio and the curvature.
\end{abstract}
\end{frontmatter}
\section{Introduction}

Set function optimization is an active field of research which is used in a broad range of applications including video summarization in machine learning~\citep{bian2017guarantees,pmlr-v80-chen18b}, splice site detection in computational biology~\citep{elenberg2018restricted,pmlr-v80-chen18b}, actuator and sensor placement problems in control theory~\citep{clark2017submodularity,clark2012leader,guo2019actuator2,guo2019actuator,EffortBounds}, multi-robot task allocation problems in robotics~\citep{gerkey2004formal,jorgensen2017matroid,zhou2018resilient,daniel2021thesis,daniel2021multi,shin2010uav} and many others. In this letter, we study the following case: the problem of minimizing an increasing nonsubmodular and nonsupermodular set function (or equivalently, maximizing a decreasing set function) over the base of a matroid. The objective and the constraints of this problem are general enough to model many of these application instances for which NP-hardness results are available in the literature~\citep{wolsey1999integer}. Thus, it  is  desirable  to  obtain  scalable  algorithms  with  provable suboptimality bounds.

Many studies have adopted greedy heuristics, because of their polynomial-time complexity and the performance bounds they are equipped with~\cite{nemhauser1978analysis,fisher1978analysis,conforti1984submodular,sviridenko2017optimal}. For the forward greedy algorithm applied to our setting, \cite[Theorem 7]{sviridenko2017optimal} provides a performance guarantee based on the notion of \textit{strong}\footnote{We use the term strong curvature to differentiate between the definition of curvature from \citep{sviridenko2017optimal} and the definitions in our work and in~\citep{bian2017guarantees}.} curvature, describing how modular the objective is.
As an alternative, \cite[Theorem 2]{guo2019actuator} provides a guarantee based on the weaker notions of submodularity ratio (describing how submodular the objective is) and curvature (describing how supermodular the objective is). However, this guarantee scales with the cardinality of the ground set. To the best of our knowledge, there exists no forward greedy guarantee applicable to our setting utilizing both submodularity ratio and curvature, simultaneously, which is also problem-size independent. This will be the first goal of this letter.\looseness=-1

An inherent drawback of the forward greedy algorithm is that any performance guarantee has to involve the objective evaluated at the empty set as the reference value. This reference is known to have an undesirable effect in several applications~\citep{zhang2011adaptive,tropp2004greed,chrobak2006reverse,guo2019actuator}. 
An alternative is to adopt the reverse greedy, which excludes the least desirable elements iteratively starting from the full set. In this case, any potential performance guarantee would instead involve the objective function evaluated at the full set, which might be a more preferable reference point. For the reverse greedy applied to our setting, \cite[Theorem 6]{sviridenko2017optimal} provides a performance guarantee again based on the notion of strong curvature. When only cardinality constraints are present, \citep[Theorem~1]{bian2017guarantees} is applicable and it provides a guarantee based on the weaker notions of submodularity ratio and curvature. 
However, to the best of our knowledge, there exists no reverse greedy guarantee applicable to our setting utilizing both curvature and submodularity ratio, simultaneously. This will be the second goal of this letter.

Our contributions are as follows. We obtain a performance guarantee for the forward greedy algorithm applied to minimizing increasing nonsubmodular and nonsupermodular set functions, characterized by submodularity ratio and curvature, over the base of a matroid. This result is presented in Theorem~\ref{thm:fwd}. For the same setting, we then obtain a performance guarantee for the reverse greedy algorithm, see Theorem~\ref{thm:rev}.
Both guarantees utilize a version of the ordering property derived in \citep[Lemma~1]{ilev18}, which is inspired by its continuous polymatroid counterpart from \citep[Theorem~6.1]{conforti1984submodular}. For both guarantees, we demonstrate more efficient greedy notions of the curvature and the submodularity ratio, since the original definitions are computationally intractable. Finally, we provide a comparison of these two theoretical performance guarantees for different values of submodularity ratio and curvature.

In the remainder, Section~\ref{sec:prelim} introduces preliminaries and problem formulation. Sections~\ref{sec:fwdg} and~\ref{sec:rwg} study the forward and  the  reverse  greedy  algorithms,  respectively,  and  derive guarantees. Section~\ref{sec:compar} presents the comparison.

\section{Preliminaries}\label{sec:prelim}

\subsection{Properties of set functions and set constraints}

We introduce well-studied notions from combinatorial optimization literature~\citep{il2003hereditary,schrijver2003combinatorial,oxley2006matroid,welsh2010matroid}. 

Let $V$ be a finite ground set and $f: 2^V \rightarrow \R$ be our set function. In literature, function is called normalized whenever $f(\emptyset)=0$, however, assume this is not necessarily the case. For notational simplicity, we use $j$ and $\{j\}$ interchangeably for singleton sets.

\begin{definition}\label{def:monotonicity}
Function $f$ is \textit{increasing} if $f(S) \leq f(R),$ for all $S \subseteq R \subseteq V$. Function $-f$ is then \textit{decreasing}. If the inequality is strict whenever $S\subsetneq R$, then $f$ is \textit{strictly increasing} and $-f$ is \textit{strictly decreasing}.
\end{definition}

\begin{definition}\label{def:discretederivative}
 For any $S \subseteq V$ and $j\in V$,  \textit{discrete derivative} of $f$ at $S$ with respect to $j$ is given by $\rho_j(S) = f(S\cup j)-f(S)$. 
\end{definition}

If $j\in S$, we have $\rho_j(S)=0$.
For any $R\subseteq V$, we generalize the definition above to $\rho_R(S) = f(S\cup R)-f(S)$.

\begin{definition}\label{def:submodularity}
 Function $f$ is \textit{submodular} if
	  $\rho_j(R)\le\rho_j(S),$
for all $S \subseteq R \subseteq V$, for all $j \in V \setminus R$.
\end{definition}

In several practical problems, the discrete derivative diminishes as $S$ expands yielding the submodularity property, see the examples in \citep{Krause2005Near,bach2013learning}. Unfortunately, functions used in many problems, including the ones we consider, do not have this property. Instead, these problems involve increasing set functions, allowing the use of \textit{submodularity ratio} describing how far a nonsubmodular set function is from being submodular. This property was first introduced by \cite{lehmann2001combinatorial}.

\begin{definition}\label{def:submodularityratio}
The \textit{submodularity ratio} of an increasing function $f$ is the largest scalar $\gamma \in \R_+$ such that $\gamma\rho_j(R)\le\rho_j(S),$
for all $S \subseteq R \subseteq V$, for all $j \in V \setminus R$. Function $f$ with submodularity ratio $\gamma$ is called $\gamma$-submodular. 
\end{definition}

Observe that the definition above is not well-posed unless $f$ is increasing, in other words, both $\rho_j(R)$ and $\rho_j(S)$ are nonnegative.\footnote{Function $f$ is submodular, if and only if $-f$ is supermodular. Thus, one may consider the submodularity ratio of a decreasing function $f$ to be the curvature of $-f$, see Definition~\ref{def:curvature} below.}
It can easily be verified that, for an increasing function $f$, we have $\gamma \in [0,1]$ and submodularity is attained if and only if $\gamma=1$.  

We briefly review an alternative but nonequivalent submodularity ratio notion from \citep{bian2017guarantees,das2011submodular}. The \textit{cumulative submodularity ratio} is the largest scalar $\gamma' \in \R_+$ such that
$\gamma' \rho_R(S) \leq \sum_{j \in R \setminus S} \rho_j(S),$
for all $S,R \subseteq V$. The ratio $\gamma$ of Definition~\ref{def:submodularityratio} satisfies the inequalities for the definition of $\gamma'$, but the reverse argument does not necessarily hold. Hence, $\gamma \leq \gamma'$, see \cite[Appendix B]{guo2019actuator}.
This notion $\gamma'$ is generally restricted to the case when deriving guarantees for greedy heuristics with only cardinality constraints, because only then it allows to derive bounds of the form of~\citep[Lemma~1]{bian2017guarantees} for the linear programming proof from~\citep{conforti1984submodular}. Utilizing the submodularity notion as per Definition~\ref{def:submodularityratio} is needed for the guarantees we will derive.

\begin{definition}\label{def:supermodularity}
 Function $f$ is \textit{supermodular} if
	  $\rho_j(R)\ge\rho_j(S),$
for all $S \subseteq R \subseteq V$, for all $j \in V \setminus R$.
\end{definition}

Other than submodularity, another widely-used notion is supermodularity we defined above, that is, the increasing discrete derivatives property. A function which is both supermodular and submodular is modular/additive. Similar to the case with submodularity, objective functions we consider do not exhibit supermodularity as well. By introducing the \textit{curvature}, that is, how far a nonsupermodular increasing function is from being supermodular, we obtain a more precise description on how the discrete derivatives change.

\begin{definition}\label{def:curvature}
The \textit{curvature} of an increasing function $f$ is the smallest scalar $\alpha \in \R_+$ such that
$\rho_j(R) \ge (1-\alpha)  \rho_j(S) ,$
for all $S \subseteq R \subseteq V$, for all $j \in V \setminus R$. Function $f$ with curvature $\alpha$ is called $\alpha$-supermodular. 
\end{definition}

It can easily be verified that, for an increasing function $f$, we have $\alpha \in [0,1]$ and supermodularity is attained if and only if $\alpha=0$.  A cumulative definition is also applicable, however, we leave it out, see~\citep{karaca2018exploiting}.

Next, we provide two propositions regarding these ratios. The first is an observation relevant for the applications from the literature we discuss. The second will be useful when adopting the reverse greedy algorithm.

	\begin{proposition}\label{prop:bound}
	Suppose $f$ is strictly increasing and $0<\underline{f}\leq \rho_j(S)\leq\overline{f},$ for all $S\subsetneq V,$ for all $j \notin S.$ Then, we have $\gamma\geq {\underline{f}}\big{/} {\overline{f}}$ and $\alpha\leq 1-{\underline{f}}\big{/}{\overline{f}}$.
	\end{proposition}
	\begin{proof}
	Since $\rho_j(R)>0$, by reorganizing Definition~\ref{def:submodularityratio}, we obtain $\gamma=\min_{S \subseteq R \subseteq V,\ j \in V \setminus R}\dfrac{\rho_j(S)}{\rho_j(R)}.$ Clearly, the term on the right is lower bounded by ${\underline{f}}\big{/} {\overline{f}}$. Similarly, we can reorganize Definition~\ref{def:curvature} as follows $ \dfrac{1}{(1-\alpha)} =\max_{S \subseteq R \subseteq V,\ j \in V \setminus  R}\dfrac{\rho_j(S)}{\rho_j(R)}$. The term on the right is upper bounded by ${\overline{f}}\big{/} {\underline{f}}.$ A simple manipulation of the inequality gives us the desired result.\hfill$\QEDA$
	\end{proof}
	\begin{proposition}\label{prop:rev}
	Let $\hat{f}(S)=-f(V\setminus S),$ for all $S\subseteq V.$ Let $\hat\gamma$ and  $\hat\alpha$ be the submodularity ratio and the curvature of function $\hat{f}$, respectively. Then, we have
	$\hat\gamma=1-\alpha$ and $\hat\alpha=1-\gamma$.
	\end{proposition}
\begin{proof}
Observe that function $\hat{f}$ is increasing, hence submodularity ratio and curvature are well-defined.
Let $\hat{\rho}_j(S) = \hat{f}(S\cup j)-\hat{f}(S)={f}(V\setminus S)-{f}(\{V\setminus S\}\setminus j\}),$ for all $S,j$. 
The submodularity ratio of function $\hat{f}$ is the largest scalar $\hat{\gamma}$ such that $\hat{\gamma}\hat{\rho}_j(R)\le\hat{\rho}_j(S),$
for all $S \subseteq R \subseteq V$, for all $j \in V \setminus R$. Using the definition of $\hat{\rho}_j$, this ratio is also the largest scalar $\hat{\gamma}$ such that $ \hat{\gamma}  \rho_j(S)\le \rho_j(R),$
for all $S \subseteq R \subseteq V$, for all $j \in V \setminus R$. Thus, 	$\hat\gamma=1-\alpha$ by Definition~\ref{def:curvature}. Using a similar reasoning, one would also obtain $\hat\alpha=1-\gamma$.
\hfill$\QEDA$
\end{proof}

Many combinatorial optimization problems from the literature are subject to constraints that are more complex than simple cardinality constraints, see the examples in~\citep{krause2011submodular,tzoumas2018resilient}. Among those, we introduce matroids. They will capture problems of interest in Section~\ref{sec:probform}. Moreover, they are known to allow performance guarantees for greedy heuristics thanks to their specific properties outlined below~\citep{edmonds1971matroids,welsh1970matroid}.

\begin{definition}\label{def:matroid}
A \textit{matroid} $\mathcal{M}$ is an ordered pair $(V,\mathcal{F})$ consisting of a ground set $V$ and a collection $\mathcal{F}$ of subsets of $V$ which satisfies \begin{itemize}
    \item[(i)] $\emptyset \in \mathcal{F}$, \item[(ii)] if $S,R \in \mathcal{F}$ and $R\subsetneq S$, then $R\in \mathcal{F}$, \item[(iii)] if $S_1$,$S_2\in \mathcal{F}$ and $|S_1|<|S_2|$, there exists $j \in S_2\setminus S_1$ such that $j\cup S_1\in \mathcal{F}$.
\end{itemize} Every set in $\mathcal{F}$ is called \textit{independent}. Maximum independent sets refer to those with the largest cardinality, and they are called the bases of a matroid.
\end{definition}

Clearly, all bases have the same cardinality by property \textit{(iii)}.
This last property of a matroid is considered as the generalization of the linear independence relation from linear algebra. Intuitively, this property will later let us to keep track of the elements that the greedy algorithm is missing from the optimal solution. 

To adopt the reverse greedy algorithm, an additional concept will be required, that is, the dual of a matroid.

\begin{definition}
\label{def:dual of a matroid}
Given a matroid $\mathcal{M}=(V,\mathcal{F})$, let $\hat{\mathcal{F}} = \{U\text{ }|\text{ }\exists \text{ a base}$ $ M\in \text{  $\mathcal{F}$ such}$ $\text{that } U\subseteq V\setminus M\}$. The pair $\hat{\mathcal{M}}=(V,\hat{\mathcal{F}})$ is called the \textit{dual} of the matroid $(V,\mathcal{F})$. 
\end{definition}

The pair $\hat{\mathcal{M}}=(V,\hat{\mathcal{F}})$ satisfies all the axioms of a matroid.
Suppose $\{M_i\}_{i=1}^q$ is the collection of all bases of matroid $(V,\mathcal{F})$. Then, $\{V\setminus{M_i}\}_{i=1}^q$ defines the collection of all bases for the dual $(V,\hat{\mathcal{F}})$, see \citep[Ch. 2]{welsh2010matroid} and \citep[Lemma 1]{guo2019actuator}.

\subsection{Problem formulation}\label{sec:probform}

Our goal is to solve
\begin{equation}
\begin{aligned}
& \underset{S\subseteq V}{\min}\
 f(S),\;\text{increasing,}\;\text{$\gamma$-submodular},\;\text{$\alpha$-supermodular} \\
&\ \mathrm{s.t.}\ S\in \mathcal{F}\text{,\, $\mathcal{M}=(V,\mathcal{F})$ is a matroid,}\, |S|= N, 
\end{aligned}
\label{eq:generalized}
\end{equation}
where the cardinality of any base of $\mathcal{M}$ is given by $N\in\Z_+$. The guarantees we derive will be applicable as long as the cardinality of any base of $\mathcal{M}$ is larger than or equal to $N$. Note that if not, the problem would be infeasible. We can reformulate such problems as \eqref{eq:generalized}, since the intersection of a uniform matroid (i.e., $\{S\subseteq V\,\rvert\,|S|\le Q\},$ where $Q\in\Z_+,$ $Q\le |V|$) and any matroid results in another matroid. 
Finally, let $S^*$ denote the optimal solution of \eqref{eq:generalized}. 
\looseness=-1

Let us briefly explain three of the well-studied applications the problem above can model. 

$\bullet$ In multi-robot task allocation, the goal is to distribute a given set of tasks to a fleet of robots. The ground set $V$ is the set of all robot-task pairs, and the constraints are given by a partition matroid and a cardinality constraint. The objective is a measure defined over the set partitions. For instance, in~\citep{daniel2021multi}, the goal is to maximize the success probability of a mission which decreases when more and more tasks are allocated to the robots. This measure is both nonsubmodular and nonsupermodular. Observe that such problems can equivalently be stated as the minimization of an increasing function as in~\eqref{eq:generalized}.

$\bullet$ In actuator placement problem,
the goal is to select  a  subset  from a  finite  set  of  possible  placements when designing controllers over a large-scale network.
The ground set $V$ is the set of all possible placements. The constraints can be a cardinality constraint on the number of actuators, and also a controllability requirement (a matching matroid). The objective is then a desired network performance metric. For instance, in~\citep{guo2019actuator}, the goal is to minimize an energy consumption metric, which increases as we pick more and more actuators to remove from the network (the so-called actuator removal problem). Thus, this problem maps to~\eqref{eq:generalized}. Such metrics are well-known to be nonsubmodular and nonsupermodular, and the literature obtains bounds on the submodularity ratio and the curvature based on versions of Proposition~\ref{prop:bound} utilizing eigenvalue inequalities, see \citep[Propositions 1 and 2]{summers2017performance}. 

$\bullet$ In video summarization, 
the goal is to to pick a few frames from a video which summarize it. The ground set $V$ is the set of all frames, and the constraints are both a cardinality constraint and a partition constraint, for instance, to extract one representative frame from every minute.
The objective is a summary quality measure defined over a set of frames.
For instance, in~\citep{pmlr-v80-chen18b}, the goal is to maximize an objective that favors subsets with higher diversity. This problem can easily be reformulated such that it maps to~\eqref{eq:generalized}.

\section{Greedy algorithms}
\subsection{Forward greedy algorithm and the performance guarantee}\label{sec:fwdg}
For Algorithm~\ref{alg:fwd}, the following definitions and explanations are in order. At iteration $t$, the forward greedy algorithm chooses $s_t := S^t\setminus S^{t-1}$ with the corresponding discrete derivative $\rho_{t} := f(S^{t})-f(S^{t-1})$. In Line~5, we have a matroid feasibility check. We assume that this can be done in polynomial-time, which is the case for all the applications we discussed above, e.g., \cite[\S 5.B\&C]{guo2019actuator}. Thanks to the properties of a matroid, we do not need to reconsider an element that has already been rejected by the feasibility check. To this end, the set $U^t$ denotes the set of elements having been considered by the matroid feasibility check before choosing $s_{t+1}$. The final forward greedy solution is $S^\mathsf{f}:=S^N$, and it is a base of  $(V,\mathcal{F})$, since it lies in $\mathcal{F}$ and has cardinality $N$ by the properties of a matroid.

\begin{algorithm}[t]
        \caption{Forward Greedy Algorithm}
        \label{alg:fwd}
        \begin{algorithmic}[1]
        \Require Set function $f$, ground set $V$, matroid $(V,\mathcal{F})$, cardinality constraint $N$
        \Ensure Forward greedy solution $S^\mathsf{f}$
        \Function {ForwardGreedy}{$f, V,{\mathcal{F}}, N$}
        
        \State {$S^0 = \emptyset$, $U^0 = \emptyset$, $t=1$}
        \While{$|S^{t-1}|<N$}
            \State ${j^*(t)}= \arg\min_{j\in V\setminus U^{t-1}}\rho_j(S^{t-1})$
            \If {$S^{t-1}\cup j^*(t)\notin\mathcal{\mathcal{F}}$}
        \State $U^{t-1}\gets U^{t-1}\cup j^*(t)$
        \Else
        \State $\rho_{t} \gets\rho_{j^*(t)}(S^{t-1})$ and $s_t = j^*(t)$
        \State $S^{t} \gets S^{t-1}\cup j^*(t)$ and $U^{t}\gets U^{t-1}\cup j^*(t)$
        \State $t\gets t+1$
    \EndIf    
            \EndWhile
            \State $S^\mathsf{f} \gets S^N $
        \EndFunction
        \end{algorithmic}
    \end{algorithm}

Main result is shown in the following theorem.

\begin{theorem}
\label{thm:fwd}
If Algorithm \ref{alg:fwd} is applied to \eqref{eq:generalized}, then
\begin{equation}
\label{eq:finalBound}
\frac{f(S^\mathsf{f})-f(\emptyset)}{f(S^*)-f(\emptyset)} \leq \dfrac{1}{\gamma(1-\alpha)}.
\end{equation}
\end{theorem}

We first need the following lemma.
\begin{lemma}\label{lem:ord}
   For any base $M\in \mathcal{F}$, the elements of $M=\{m_1,\ldots,m_N\}$ can be ordered so that $$\rho_{m_t}(S^{t-1})\geq\rho_{t}=\rho_{s_t}(S^{t-1}),$$ holds for $t=1,\ldots,N.$  Moreover, whenever $s_t\in M$, we have that $m_t=s_t$.
\end{lemma}
\begin{proof}
For this proof, we extend a method from a similar ordering property which is derived when applying greedy heuristics on dependence systems~\citep[Lemma~1]{ilev18} (specifically on comatroids, which are the complementary notion of matroids in independence systems, see also~\citep{il2003hereditary,il2006performance, karaca2019}). \citep[Lemma~1]{ilev18} was originally inspired by a study on greedy heuristics over integral polymatroids (that is, a continuous extension of matroids), see \citep[Theorem~6.1]{conforti1984submodular}.

We prove by induction. Assume the elements $\{m_{t+1},\ldots, m_N \}$ are found. Let $\tilde{M}_t=M\setminus \{m_{t+1},\ldots, m_N \}$.
If $s_t\in \tilde{M}_t$, we let $m_t=s_t,$ and the condition is satisfied, and the second statement of the lemma holds. If $s_t\not\in \tilde{M}_t$, by property~\textit{(iii)} of Definition~\ref{def:matroid} and $\tilde{M}_t\in\mathcal{F}$, we know that there exists $j \in \tilde{M}_t\setminus S^{t-1}$ such that $j\cup S^{t-1}\in \mathcal{F}$. Moreover, $\rho_{j}(S^{t-1})\geq\rho_{s_t}(S^{t-1})$, since $j$ is not the element chosen by the greedy algorithm. Hence, we pick $m_t=j.$ The existence of $m_N$ follows from property~\textit{(iii)} of Definition~\ref{def:matroid}: there exists $m_N\in M\setminus S^{N-1}$ such that $m_N\cup S^{N-1}\in \mathcal{F}$. This concludes the proof. \hfill$\QEDA$
\end{proof}

This lemma plays a significant role in obtaining suboptimality bounds, which becomes clear once we pick $M$ to be the optimal solution $S^*$ in the following proof of Theorem~\ref{thm:fwd}.

\begin{proof}[Proof of Theorem~\ref{thm:fwd}]
Let $S^*=\{s^*_1,\ldots,s^*_N\}$, where the elements $s^*_t$ are ordered according to Lemma~\ref{lem:ord}. Let $S^*_t=\{s^*_1,\ldots,s^*_t\}$ for $t=1,\ldots,N$, and $S^*_0=\emptyset$. Using this definition, we obtain 
\begin{equation}\label{eq:lwb} f(S^*)-f(\emptyset) = \sum_{t=1}^N \rho_{s^*_t}(S^*_{t-1})\geq(1-\alpha) \sum_{t=1}^N \rho_{s^*_t}(\emptyset).
\end{equation}
The equality follows from a telescoping sum. The inequality follows from Definition~\ref{def:curvature}. On the other hand, a similar observation can also be made for the forward greedy solution as follows
\begin{equation*}
    \begin{split}
        f(S^\mathsf{f})-f(\emptyset) &= \sum_{t=1}^N \rho_t\\&=\sum_{t:s_t\in S^*} \rho_{s_t}(S^{t-1})+\sum_{t:s_t\notin S^*} \rho_{s_t}(S^{t-1}),
    \end{split}
\end{equation*} 
where the last equality decomposes the greedy steps into those that coincide with the optimal solution and those that do not. Invoking Lemma~\ref{lem:ord} for the term on the right-hand side, we obtain
\begin{equation}\label{eq:setref}
    f(S^\mathsf{f})-f(\emptyset)\leq \sum_{t:s_t\in S^*} \rho_{s_t}(S^{t-1})+\sum_{t:s_t^*\notin S^\mathsf{f}} \rho_{s_t^*}(S^{t-1}),
\end{equation}
where the term on the right also utilizes $\{t:s_t\notin S^*\}=\{t:s_t^*\notin S^\mathsf{f}\}$, which is a direct result of the last statement of  Lemma~\ref{lem:ord}.
Now, notice that for any $s_t^*\notin S^\mathsf{f}$ we have $s_t^*\notin S^{t-1}$. Hence using Definition~\ref{def:submodularityratio}, we obtain 
$$f(S^\mathsf{f})-f(\emptyset)\leq \sum_{t:s_t\in S^*} \rho_{s_t}(S^{t-1})+\dfrac{1}{\gamma  }\sum_{t:s_t^*\notin S^\mathsf{f}} \rho_{s_t^*}(\emptyset).$$
Next, the term on the left can also be upper bounded by utilizing Definition~\ref{def:submodularityratio}, giving us
\begin{equation}\label{eq:up}
    \begin{split}
    f(S^\mathsf{f})-f(\emptyset)&\leq \dfrac{1}{\gamma  }\sum_{t:s_t\in S^*} \rho_{s_t}(\emptyset)+\dfrac{1}{\gamma  }\sum_{t:s_t^*\notin S^\mathsf{f}} \rho_{s_t^*}(\emptyset)\\
    &=\dfrac{1}{\gamma}\sum_{t:s_t\in S^*} \rho_{s_t}(\emptyset)+\dfrac{1}{\gamma  }\sum_{t:s_t\notin S^*} \rho_{s_t^*}(\emptyset)\\
    &=\dfrac{1}{\gamma  }\sum_{t=1}^N \rho_{s^*_t}(\emptyset),
   \end{split}
\end{equation} 
The first equality reapplies the set reformulation $\{t:s_t\notin S^*\}=\{t:s_t^*\notin S^\mathsf{f}\}$ (previously found in \eqref{eq:setref}), and the second equality combines the two summations. Finally, we can combine \eqref{eq:lwb} and \eqref{eq:up}, to obtain 
\begin{equation*}
\frac{f(S^\mathsf{f})-f(\emptyset)}{f(S^*)-f(\emptyset)} \leq \dfrac{1}{\gamma(1-\alpha)}.
\end{equation*}
This completes the proof of theorem.\hfill\QEDA
\end{proof}

\cite[Theorem 2]{guo2019actuator} offers a guarantee for this setting. This is given by $ \frac{\gamma}{1-\gamma}\big((2N+1)^{\frac{1-\gamma}{\gamma(1-\alpha)}}-1\big).$ In contrast, the guarantee in Theorem~\ref{thm:fwd} above is independent of the problem size, and also tighter for any of the $(\alpha, \gamma, N)$ pairs.\footnote{\citep[Propositions 4 and 5]{guo2019actuator} prove that there is no performance guarantee for the forward greedy algorithm unless both submodular-like and supermodular-like properties are present in the objective function. Observe that this is confirmed by Theorem~\ref{thm:fwd}.}

As an alternative, \cite[Theorem 7]{sviridenko2017optimal} offers another guarantee, which is given by $1/(1-c)$, where the strong curvature $c$ quantifies how far a function is from being modular: the smallest parameter $c\in[0,\,1]$ such that $\rho_j(R) \ge (1-c)  \rho_j(S) ,$
for all $S,R \subseteq V\setminus j.$ This novel notion is a significantly stronger requirement than having both the submodularity ratio and the curvature, simultaneously~\citep{guo2019actuator}. Hence, it is not possible to compare it with our guarantee other than the case of a modular objective, that is, $c=0,\,\gamma=1,\,\alpha=0$. For both guarantees, modularity confirms the optimality of the forward greedy algorithm, as it is well-established by the Rado-Edmonds theorem~\citep{welsh1970matroid}. Note that computing the strong curvature notion requires an exhaustive enumeration of all inequalities in its definition, since the proof method of~\cite[Theorem 7]{sviridenko2017optimal}  does not allow any greedy curvature computation as we present below.

\begin{corollary}\label{cor:fwdgr}
Let $\gamma^\mathsf{fg}$ be the largest $\tilde\gamma$ that satisfies $\tilde\gamma\rho_{s}(S)\leq\rho_{s}(\emptyset)$ for all $S\in\mathcal{F}$, $|S|\leq N-1$ and $S\cup s \in \mathcal{F}$.
Then, $\gamma^{\mathsf{fg}}$ is called the \textit{forward greedy submodularity ratio} with $\gamma^{\mathsf{fg}}\geq \gamma$. Let $\alpha^\mathsf{fg}$ be the smallest $\tilde\alpha$ that satisfies $\rho_{s}(S)\geq(1-\tilde\alpha)\rho_{s}(\emptyset)$ for all $S\in\mathcal{F}$, $|S|\leq N-1$ and $S\cup s \in \mathcal{F}$.
Then, $\alpha^{\mathsf{fg}}$ is called the \textit{forward greedy curvature} with $\alpha^{\mathsf{fg}}\leq \alpha$. The performance guarantee can then be written as 
\begin{equation*}\label{eq:corbound}
\begin{split}
&\frac{f(S^\mathsf{f})-f(\emptyset)}{f(S^*)-f(\emptyset)} \leq \dfrac{1}{\gamma^{\mathsf{fg}}(1-\alpha^{\mathsf{fg}})},\, \text{or equivalently},\\
&f(S^\mathsf{f})\leq\dfrac{1}{\gamma^{\mathsf{fg}}(1-\alpha^{\mathsf{fg}})}f(S^*)+\Bigg(1-\dfrac{1}{\gamma^{\mathsf{fg}}(1-\alpha^{\mathsf{fg}})}\Bigg)f(\emptyset).
\end{split}
\end{equation*}
\end{corollary}

The forward greedy submodularity ratio and the forward greedy curvature can be obtained after analyzing $\mathcal{O}(\textstyle\binom{|V|}{N})$ inequalities, which could still be large. However, they are significantly more tractable than the original definitions. Since $\gamma^{\mathsf{fg}}\geq \gamma$ and $\alpha^{\mathsf{fg}}\leq \alpha$, the performance guarantee in Corollary~\ref{cor:fwdgr} can essentially be better than the one in Theorem~\ref{thm:fwd}. Notice that $(\gamma^{\mathsf{fg}},\alpha^\mathsf{fg})$ changes with the constraint set of the problem since the inequalities defining $(\gamma^{\mathsf{fg}},\alpha^\mathsf{fg})$ would then be different. In contrast, submodularity ratio and curvature depend only on the objective function.
 
The performance guarantee in Corollary~\ref{cor:fwdgr} can still be loose, because of the reference value~$f(\emptyset).$ For instance, in multi-robot task allocation problems, $f(\emptyset)$ corresponds to minus the safety of a plan with no tasks, see~\citep{daniel2021multi}. In such applications, the values $f(\emptyset)\approx -1$ and $(1-1/[{\gamma^{\mathsf{fg}}(1-\alpha^{\mathsf{fg}})}])< 0$ can make the bound in~\eqref{eq:corbound} large.
 
In the next section, we consider a variant of the greedy algorithm that comes along with a performance guarantee that does not depend on~$f(\emptyset).$

\subsection{Reverse greedy algorithm and the performance guarantee}\label{sec:rwg}
For Algorithm~\ref{alg:rev}, the following definitions and explanations are in order. For compactness, we define a \textit{shifted} discrete derivative $\delta_j(S):=\rho_j(S\setminus j)=f(S)-f(S\setminus j),$ for all $S\subseteq V$, $j\in S$. At iteration $t$, the reverse greedy algorithm chooses $r_t := X^{t-1}\setminus X^t$ to remove, with the maximal reduction  $\delta_{t} := f(X^{t-1})-f(X^{t})$. 
In Line 5, we have a matroid feasibility check. In contrast to Algorithm~\ref{alg:fwd}, this matroid feasibility check requires that our intermediate solutions are supersets of a base of the matroid $\mathcal{M}.$  
The set $Y^t$ denotes the set of elements having been considered by the matroid feasibility check before choosing $x_{t+1}$. The final reverse greedy solution is $S^\mathsf{r} := X^{|V|-N} $, and it is a base of  $(V,\mathcal{F})$, since it lies in $\mathcal{F}$ and has cardinality $N$ by the properties of a matroid.

\begin{algorithm}[t]
        \caption{Reverse Greedy Algorithm}
        \label{alg:rev}
        \begin{algorithmic}[1]
        \Require Set function $f$, ground set $V$, matroid $(V,\mathcal{F})$, cardinality constraint $N$
        \Ensure Reverse greedy solution $S^\mathsf{r}$
        \Function {ReverseGreedy}{$f, V,{\mathcal{F}}, N$}
        
        \State {$X^0 = V$, $Y^0 = \emptyset$, $t=1$}
        \While{$|X^{t-1}|>N$}
            \State ${k^*(t)}= \arg\max_{j\in V\setminus Y^{t-1}}\delta_j(X^{t-1})$
            \If {$\exists M\in\mathcal{F}$ such that $M\subseteq\left\{X^{t-1}\setminus k^*(t)\right\}$ and $|M|=N$}
        \State $Y^{t-1}\gets Y^{t-1}\cup k^*(t)$
        \Else
        \State $\delta_{t} \gets\delta_{k^*(t)}(X^{t-1})$ and $r_t = k^*(t)$
        \State $X^{t} \gets X^{t-1}\setminus k^*(t)$ and $Y^{t}\gets Y^{t-1}\cup k^*(t)$
        \State $t\gets t+1$
    \EndIf    
            \EndWhile
            \State $S^\mathsf{r} \gets X^{|V|-N} $
        \EndFunction
        \end{algorithmic}
    \end{algorithm}
    
    Main result is shown in the following theorem.
    
\begin{theorem}
\label{thm:rev}
If Algorithm~\ref{alg:rev} is applied to \eqref{eq:generalized}, then
\begin{equation*}
\label{eq:finalBound2}
\frac{f(V)-f(S^\mathsf{r})}{f(V)-f(S^*)} \geq \dfrac{1-\alpha}{1+(1-\gamma)(1-\alpha)}.
\end{equation*}
\end{theorem}

For the sake of clarity of the notation, our proof will utilize the following reformulation of~\eqref{eq:generalized}: 
\begin{equation}
\begin{aligned}
& \underset{R\subseteq V}{\max}\
 \hat{f}(R):=-f(V\setminus R),\;\text{increasing,}\;\text{$(1-\alpha)$-submodular},\;\text{$(1-\gamma)$-supermodular} \\
&\ \mathrm{s.t.}\ R\in \hat{\mathcal{F}}\text{,\, $\hat{\mathcal{M}}=(V,\mathcal{F})$ is a matroid,}\, |R|= |V|-N=\hat{N}, 
\end{aligned}
\label{eq:generalizedref}
\end{equation}where the cardinality of any base of $\hat{\mathcal{M}}$ is given by $|V|-N=\hat{N}\in\Z_+$.
The equivalence of~\eqref{eq:generalized} and~\eqref{eq:generalizedref} follows directly from Proposition~\ref{prop:rev} and Definition~\ref{def:dual of a matroid} (that is, the definition of the dual matroid). Denote its optimal solution by~$R^*$. Clearly, we have $R^*=V\setminus S^*$.

Forward greedy algorithm applied to~\eqref{eq:generalizedref} is presented in Algorithm~\ref{alg:fwdrev}, where we define $\hat{\rho}_j(R) = \hat{f}(R\cup j)-\hat{f}(R)$, for all $R \subseteq V$ and $j\in V$. Iterations of this algorithm coincide with those of the reverse greedy algorithm applied to \eqref{eq:generalized}. Denote the forward greedy iterates by $R^{t}.$ At any iteration, we have $X^{t}=V\setminus R^{t}$. At iteration $t$, the forward greedy algorithm chooses $r_t := R^t\setminus R^{t-1}$ with the corresponding discrete derivative $\hat{\rho}_{t} := \hat{f}(S^{t})-\hat{f}(S^{t-1})$.

\begin{algorithm}[t]
        \caption{Forward Greedy Reformulation of Reverse Greedy Algorithm}
        \label{alg:fwdrev}
        \begin{algorithmic}[1]
        \Require Set function $\hat{f}$, ground set $V$, dual matroid $(V,\hat{\mathcal{F}})$, cardinality constraint $\hat{N}$
        \Ensure Reverse greedy solution $S^\mathsf{r}$
        \Function {ReverseGreedyReformulated}{$f, V,\hat{\mathcal{F}}, \hat{N}$}
        
        \State {$R^0 = \emptyset$, $Y^0 = \emptyset$, $t=1$}
        \While{$|R^{t-1}|<{\hat{N}}$}
            \State ${k^*(t)}= \arg\max_{j\in V\setminus Y^{t-1}}\hat{\rho}_j(R^{t-1})$
            \If {$R^{t-1}\cup k^*(t)\notin\hat{{\mathcal{F}}}$}
        \State $Y^{t-1}\gets Y^{t-1}\cup k^*(t)$
        \Else
        \State $\hat{\rho}_{t} \gets\hat{\rho}_{k^*(t)}(R^{t-1})$ and $r_t = k^*(t)$
        \State $R^{t} \gets R^{t-1}\cup k^*(t)$ and $Y^{t}\gets Y^{t-1}\cup k^*(t)$
        \State $t\gets t+1$
    \EndIf    
            \EndWhile
            \State $S^\mathsf{r} \gets V\setminus R^{\hat{N}} $
        \EndFunction
        \end{algorithmic}
    \end{algorithm}
    
    With the observations above in mind, we bring the ordering lemma.
    \begin{lemma}\label{lem:ord2}
    For any base $M\in \hat{\mathcal{F}}$, the elements of $M=\{m_1,\ldots,m_{\hat{N}}\}$ can be ordered so that $$\hat{\rho}_{m_t}(R^{t-1})\leq\hat{\rho}_{t}=\hat{\rho}_{r_t}(R^{t-1}),$$ holds for $t=1,\ldots,\hat{N}.$  Moreover, whenever $r_t\in M$, we have that $m_t=r_t$.
    \end{lemma}
    \begin{proof} The proof is a reformulation of that of~Lemma~\ref{lem:ord}, by changing the greedy minimization to a greedy maximization. We prove by induction. Assume the elements $\{m_{t+1},\ldots, m_{\hat{N}} \}$ are found. Let $\tilde{M}_t=M\setminus \{m_{t+1},\ldots, m_{\hat{N}} \}$.
If $r_t\in \tilde{M}_t$, we let $m_t=r_t,$ and the condition is satisfied, and the second statement of the lemma holds. If $r_t\not\in \tilde{M}_t$, by property~\textit{(iii)} of Definition~\ref{def:matroid} and $\tilde{M}_t\in\hat{\mathcal{F}}$, we know that there exists $j \in \tilde{M}_t\setminus R^{t-1}$ such that $j\cup R^{t-1}\in \mathcal{F}$. Moreover, $\hat{\rho}_{j}(R^{t-1})\leq\hat{\rho}_{r_t}(R^{t-1})$, since $j$ is not the element chosen by the greedy algorithm. Hence, we pick $m_t=j.$ The existence of $m_{\hat{N}}$ follows from property~\textit{(iii)} of Definition~\ref{def:matroid}: there exists $m_{\hat{N}}\in M\setminus R^{{\hat{N}}-1}$ such that $m_{\hat{N}}\cup R^{{\hat{N}}-1}\in \hat{\mathcal{F}}$. This concludes the proof. \hfill\QEDA
    \end{proof}
    
    We are now ready to prove our theorem.
    
    \begin{proof}[Proof of Theorem~\ref{thm:rev}] Let $R^*=\{r^*_1,\ldots,r^*_{\hat{N}}\}$, where the elements $r^*_t$ are ordered according to Lemma~\ref{lem:ord2}. Let $R^*_t=\{r^*_1,\ldots,r^*_t\}$ for $t=1,\ldots,{\hat{N}}$, and $R^*_0=\emptyset$. Using this definition, we obtain 
    \begin{equation}\label{eq:upbrev}
    \begin{split}
  \hat{f}(R^{\hat{N}}\cup R^*)- \hat{f}(\emptyset)&= \hat{f}(R^{\hat{N}})- \hat{f}(\emptyset) + \sum_{t=1}^{\hat{N}} \hat{\rho}_{r^*_t}(R^{\hat{N}}\cup R^*_{t-1})\\
  &= \hat{f}(R^{\hat{N}})- \hat{f}(\emptyset) + \sum_{t:r^*_t\notin R^{\hat{N}}} \hat{\rho}_{r^*_t}(R^{\hat{N}} \cup R^*_{t-1})\\
  &\leq \hat{f}(R^{\hat{N}})- \hat{f}(\emptyset) + \dfrac{1}{1-\alpha} \sum_{t:r^*_t\notin R^{\hat{N}}} \hat{\rho}_{r^*_t}(R^{t-1})\\
  &\leq \hat{f}(R^{\hat{N}})- \hat{f}(\emptyset) + \dfrac{1}{1-\alpha} \sum_{t:r^*_t\notin R^{\hat{N}}} \hat{\rho}_{r_t}(R^{t-1}).
  \end{split}
  \end{equation}
  The first equality follows from a telescoping sum, whereas the second equality removes the zero-valued terms. The first inequality follows from the submodularity ratio of $\hat{f}$ (which is $(1-\alpha)$), whereas the second inequality follows from Lemma~\ref{lem:ord2}.
  
  A similar observation can be made to obtain a lower bound to the term $[\hat{f}(R^{\hat{N}}\cup R^*)- \hat{f}(\emptyset)]$ above as follows
    \begin{equation}\label{eq:lwbrev}
    \begin{split}
  \hat{f}(R^{\hat{N}}\cup R^*)- \hat{f}(\emptyset)&= \hat{f}(R^*)- \hat{f}(\emptyset) + \sum_{t=1}^{\hat{N}} \hat{\rho}_{r_t}(R^{t-1}\cup R^*)\\
  &\geq  \hat{f}(R^*)- \hat{f}(\emptyset) + \gamma\sum_{t=1}^{\hat{N}} \hat{\rho}_{r_t}(R^{t-1})- \gamma\sum_{t:r_t\in R^*} \hat{\rho}_{r_t}(R^{t-1})\\
  &=\hat{f}(R^*)- \hat{f}(\emptyset) + \gamma\left[\hat{f}(R^{\hat{N}})- \hat{f}(\emptyset)\right]- \gamma\sum_{t:r_t\in R^*} \hat{\rho}_{r_t}(R^{t-1})\\
  &=\hat{f}(R^*)- \hat{f}(\emptyset) + \gamma\left[\hat{f}(R^{\hat{N}})- \hat{f}(\emptyset)\right]- \gamma\sum_{t:r_t^*\in R^{\hat{N}}} \hat{\rho}_{r_t}(R^{t-1}).
  \end{split}
  \end{equation}
  The first equality follows from a telescoping sum. The inequality follows from the application of the curvature of $\hat{f}$ (which is $(1-\gamma)$) together with the observation that some of the terms in the sum $\sum_{t=1}^{\hat{N}} \hat{\rho}_{r_t}(R^{t-1}\cup R^*)$ are zero whenever $r_t\in R^*$. The second equality sums up all the terms in the telescoping sum, whereas the third equality applies $\{t:r_t\in R^*\}=\{t:r_t^*\in R^{\hat{N}}\}$, invoking the last statement of Lemma~\ref{lem:ord2}.
  
  Now, combining~\eqref{eq:lwbrev} and \eqref{eq:upbrev}, we obtain
    \begin{equation}\label{eq:revlast}
        \begin{split} (1-\gamma)\left[\hat{f}(R^{\hat{N}})- \hat{f}(\emptyset)\right] &\ge \hat{f}(R^*)- \hat{f}(\emptyset) - \dfrac{1}{1-\alpha} \sum_{t:r^*_t\notin R^{\hat{N}}} \hat{\rho}_{r_t}(R^{t-1}) -\gamma\sum_{t:r_t^*\in R^{\hat{N}}} \hat{\rho}_{r_t}(R^{t-1}) \\
        &\geq \hat{f}(R^*)- \hat{f}(\emptyset) - \dfrac{1}{1-\alpha} \sum_{t=1}^{\hat{N}} \hat{\rho}_{r_t}(R^{t-1})\\
        & = \hat{f}(R^*)- \hat{f}(\emptyset) - \dfrac{1}{1-\alpha}\left[\hat{f}(R^{\hat{N}})- \hat{f}(\emptyset)\right].
        \end{split}
    \end{equation}
    The first inequality presents  only the combination of ~\eqref{eq:lwbrev} and \eqref{eq:upbrev}. Observing that $\dfrac{1}{1-\alpha}\ge \gamma$ for any $(\alpha,\gamma)$ pair, the second inequality combines the two sums: $\{1,\ldots,{\hat{N}}\}=\{t:r^*_t\notin R^{\hat{N}}\}\cup \{t:r^*_t\in R^{\hat{N}}\}$. The last step sums all the terms involved in the telescoping sum.
    
    By reorganizing \eqref{eq:revlast}, we get 
    \begin{equation*} \frac{\hat{f}(R^{\hat{N}})-\hat{f}(\emptyset)}{\hat{f}(R^*)-\hat{f}(\emptyset)}\geq\dfrac{1-\alpha}{1+(1-\gamma)(1-\alpha)}.
\end{equation*}
    From the equivalence of the two problems and Algorithms~\ref{alg:rev} and~\ref{alg:fwdrev}, we complete the proof:
    \begin{equation*}
\frac{f(V)-f(S^\mathsf{r})}{f(V)-f(S^*)} = \frac{\hat{f}(R^{\hat{N}})-\hat{f}(\emptyset)}{\hat{f}(R^*)-\hat{f}(\emptyset)}\geq\dfrac{1-\alpha}{1+(1-\gamma)(1-\alpha)}.
\end{equation*}
\hfill\QEDA
    \end{proof}
    
    
    \citep[Theorem 6]{sviridenko2017optimal} offers a guarantee for this setting. This is given by $1-c$, where $c$ is again the strong curvature as in \citep[Theorem 7]{sviridenko2017optimal} discussed in Section~\ref{sec:fwdg}. Here, arguments similar to the case of the forward greedy algorithm can be made both on the strength of this requirement and its computational aspect. Observe that for the case of a modular objective, both guarantees confirm the optimality of the reverse greedy algorithm on the base of a matroid. 
    
    As an alternative, \citep[Theorem 1]{bian2017guarantees} offers another guarantee for this setting if the constraint is a uniform matroid, that is, only a cardinality constraint. This is given by $\dfrac{1}{1-\gamma}\left( 1- e^{-(1-\alpha)(1-\gamma)} \right)$. This guarantee is tighter than the one in~Theorem~\ref{thm:rev}, since it treats a specialized case. However, when we have exact submodularity $\gamma=1$, both guarantees still coincide $$\lim_{\gamma\rightarrow 1}\dfrac{1}{1-\gamma}\left( 1- e^{-(1-\alpha)(1-\gamma)} \right)= {1-\alpha}.$$ Moreover, both guarantees tend to $0$ as $\alpha\rightarrow 1$ independent of $\gamma$, in other words, when supermodular-like properties are not present at all. We highlight that their proof method is not applicable to general matroids. Finally, when $\gamma=0$ and $\alpha=0$, we recover the classical $\dfrac{1}{2}$ guarantee of \cite{fisher1978analysis}, since setting $\gamma=0$ can be considered to be the case when the submodularity property of the objective is completely unknown.
    
    \begin{corollary}\label{cor:revgr}
Let $\gamma^\mathsf{rg}$ be the largest $\tilde\gamma$ that satisfies $\tilde\gamma\hat{\rho}_{r_t}(R^{t-1})\leq \hat{\rho}_{r_t}(R^{t-1}\cup R)$ for all $t$, for all $R\subset V\setminus r_t$ and $\rvert R\rvert = \hat{N}.$\footnote{We remind the reader that $\hat{\rho}_j(R) = \hat{f}(R\cup j)-\hat{f}(R)={f}(V\setminus R)-{f}(\{V\setminus R\}\setminus j)$, for all $R \subseteq V$ and $j\in V$.}
Then, $\gamma^{\mathsf{rg}}$ is called the \textit{reverse greedy submodularity ratio} with $\gamma^{\mathsf{rg}}\geq \gamma$. Let $\alpha^\mathsf{rg}$ be the smallest $\tilde\alpha$ that satisfies $\hat{\rho}_{r}(R^{t-1})\geq (1-\tilde\alpha)\hat{\rho}_{r}(R^{\hat{N}}\cup R)$ for all $t$, for all $R\subset V$ and $\rvert R\rvert = t-1$, for all $r\notin R^{\hat{N}}\cup R.$ Then, $\alpha^{\mathsf{rg}}$ is called the \textit{reverse greedy curvature} with $\alpha^{\mathsf{rg}}\leq \alpha$. The performance guarantee can then be written as
    \begin{equation*}
    \begin{split}
&\dfrac{f(V)-f(S^\mathsf{r})}{f(V)-f(S^*)} \geq \dfrac{1-\alpha^{\mathsf{rg}}}{1+(1-\gamma^{\mathsf{rg}})(1-\alpha^{\mathsf{rg}})},\\
&f(S^\mathsf{r})\leq\dfrac{1-\alpha^{\mathsf{rg}}}{1+(1-\gamma^{\mathsf{rg}})(1-\alpha^{\mathsf{rg}})}f(S^*)+\Bigg(1-\dfrac{1-\alpha^{\mathsf{rg}}}{1+(1-\gamma^{\mathsf{rg}})(1-\alpha^{\mathsf{rg}})}\Bigg)f(V).
\end{split}
    \end{equation*}
    \end{corollary}

   The reverse greedy submodularity ratio and the reverse greedy curvature can be obtained in an ex-post manner after analyzing $\mathcal{O}({\hat{N}}\textstyle\binom{\rvert V\rvert-1}{\hat{N}})$ and $\mathcal{O}({\hat{N}}\textstyle\binom{\rvert V\rvert}{\hat{N}})$ inequalities, respectively. Morover, since $\gamma^{\mathsf{rg}}\geq \gamma$ and $\alpha^{\mathsf{rg}}\leq \alpha$, the performance guarantee in Corollary~\ref{cor:revgr} can essentially be better than the one in Theorem~\ref{thm:rev}. Finally, note that a large value for $f(V)$ can be crucial for the tightness of this guarantee. For instance, in the actuator removal problem of~\cite{guo2019actuator}, when we pick the full set of actuators $V$ to remove from the system, the control energy metric could be infinite.
    
\section{Comparison of the performance guarantees}\label{sec:compar}
For the sake of visualization, we let $f(\emptyset)= -1$, $f(V)=1$, $f(S^{\ast})=F^{\ast}\in[-1,1]$. Figure~\ref{fig:table} shades the area where the guarantee for the forward greedy algorithm is better than the one for the reverse greedy algorithm.
By  decreasing  the  value  of $F^{\ast}$ from $0$ to $-1$,  one  can  observe  that  the  area  where  the  forward greedy  guarantee is better, expands. When $F^{\ast}$ is small, and the function is close to being both submodular and supermodular, the forward greedy guarantee is more desirable. In fact, \cite[Propositions 4 and 5]{guo2019actuator} prove that there is no performance guarantee for the forward greedy algorithm unless both the submodularity ratio and the curvature are utilized, simultaneously.
When $F^{\ast}$ is large enough, $F^{\ast} \ge 0$, the effect of $f(\emptyset)$ on the forward greedy guarantee is more dominant, thus, the reverse greedy outperforms  the forward greedy for all $(\alpha, \gamma)$ pairs.
Note that for other values of $f(\emptyset)$ or $f(V)$, the observations can differ. 
In practice, it could be useful to implement both greedy algorithms (which can be done efficiently with polynomial time complexity) and choose the best out of the two.

\begin{figure}[th!]
    \centering
    \includegraphics[width=0.39\linewidth]{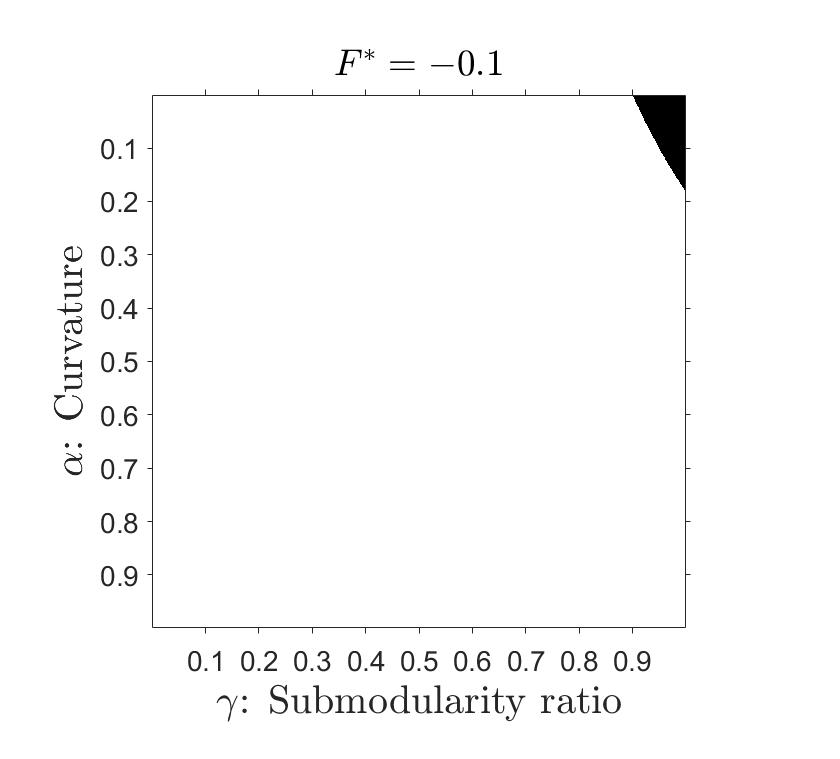}\hspace{.5cm} 
    \includegraphics[width=0.39\linewidth]{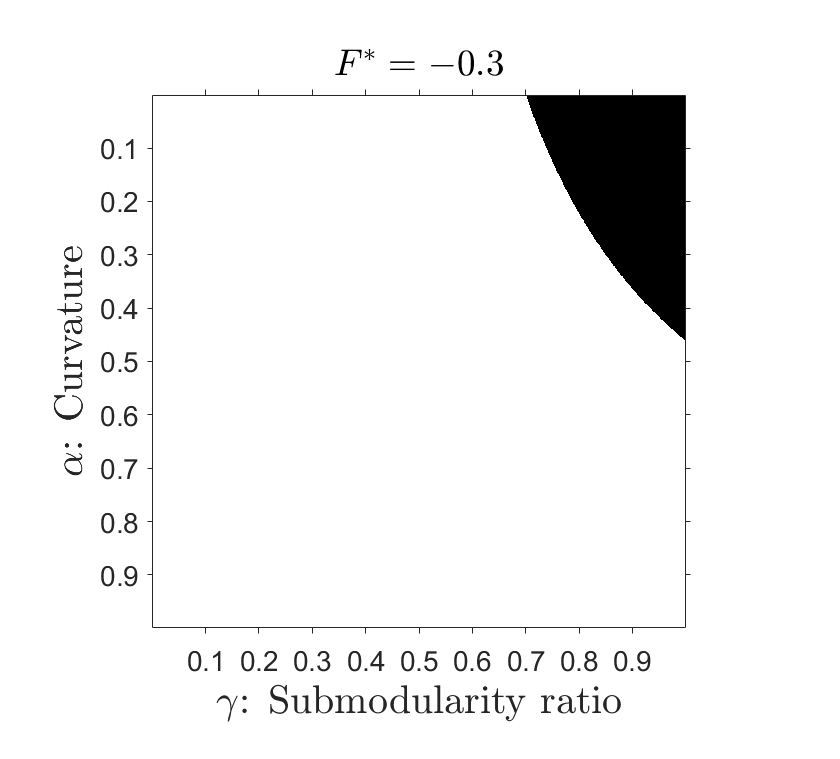}
    \includegraphics[width=0.39\linewidth]{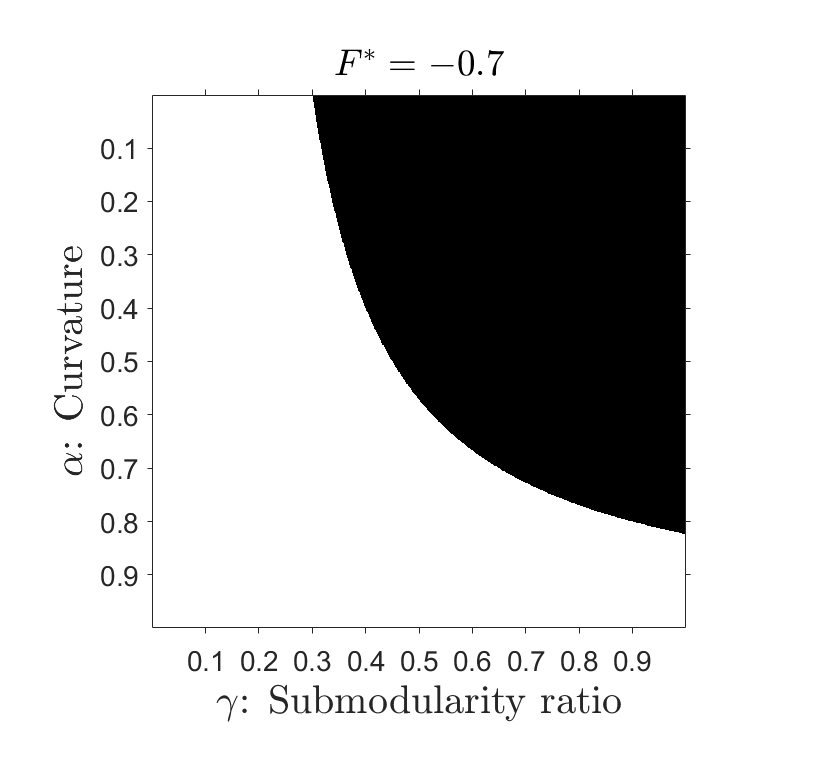}\hspace{.5cm}
    \includegraphics[width=0.39\linewidth]{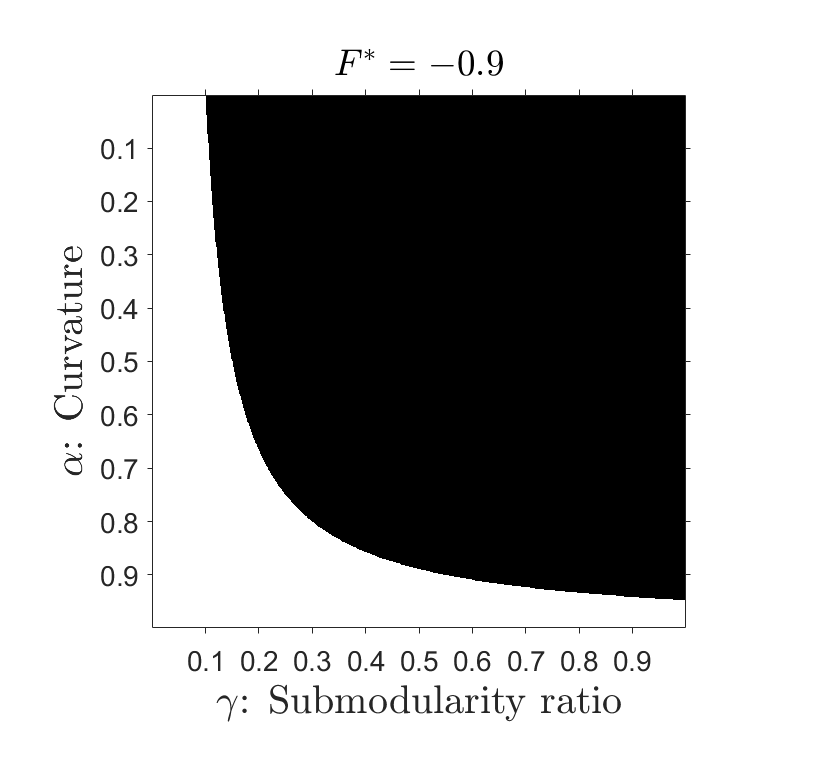}
    \caption{For each $F^*$ value, the shaded regions represent the $(\alpha, \gamma)$ pairs for which the forward greedy algorithm outperforms the reverse greedy algorithm. For $F^*\geq 0$, the reverse greedy algorithm outperforms the forward greedy algorithm for any $(\alpha, \gamma)$ pair.}
    \label{fig:table}
\end{figure}

Our future work will be focused on obtaining the problem instances for which these guarantees are potentially tight.

\section{Acknowledgments}
The authors would like to thank Professor Victor Il’ev for his helpful feedback on existing bounds on the worst-case behaviour of
greedy algorithms, Dr. Yatao An Bian for early discussions on the literature for existing properties of set functions.

\bibliographystyle{IEEEtran}
\newpage
\bibliography{library}

\end{document}